\documentclass[11pt]{amsart}
\usepackage[cm]{fullpage}
\usepackage{amssymb,amsmath,amsthm,amscd,mathrsfs,graphicx}
\usepackage[cmtip,all]{xy}
\usepackage{pb-diagram}

\numberwithin{equation}{section}

\newtheorem{theorem}{Theorem}[section]
\newtheorem{proposition}{Proposition}[section]
\newtheorem{lemma}[theorem]{Lemma}

\newcommand{\ov}[1]{\overline{#1}}

\newcommand{\ve}{\varepsilon}

\theoremstyle{definition}

\usepackage{xcolor}

\theoremstyle{remark}

\begin{document}
\bibliographystyle{amsplain}

\author[A. Chau]{Albert Chau}
\address{Department of Mathematics, The University of British Columbia, 1984 Mathematics Road, Vancouver, B.C.,  Canada V6T 1Z2.  Email: chau@math.ubc.ca. } 
\author[B. Weinkove]{Ben Weinkove}

\address{Department of Mathematics, Northwestern University, 2033 Sheridan Road, Evanston, IL 60208, USA.  Email: weinkove@math.northwestern.edu.}

\thanks{Research supported in part by  NSERC grant $\#$327637-06 and NSF grant DMS-1709544}
\title{Strong space-time convexity and the heat equation}
\begin{abstract}
We prove local strong convexity of the space-time level sets of the heat equation on convex rings for zero initial data, strengthening a result of Borell.  Our proof introduces a parabolic version of a two-point maximum principle of Rosay-Rudin.
\end{abstract}

\maketitle

\vspace{-10pt}

\section{Introduction} \label{intro}
 
A classic question in elliptic PDEs is:  does  the solution to a Dirichlet problem on a domain or convex ring inherit convexity properties from its boundary?  Building on the well-known result that the Green's function of a convex domain in $\mathbb{R}^2$ has convex level curves (see \cite{A}), this question has been studied by many authors including Gabriel, Lewis and Caffarelli-Spruck \cite{ALL, BG, BGMX, BLS, BL, CF, CGM, CS, CMY, DK, G, Gu, Ka,   Ko2, KL, L, Lo, RR, S, SWYY, SW, W}.
  One method is the ``macroscopic'' approach, using a globally defined function of two points; another is  ``microscopic'', using the principal curvatures of the level sets and a constant rank theorem.  These results show that for a large class of PDEs, the superlevel sets of the solution $u$ are convex (i.e. $u$ is \emph{quasiconcave}) if the boundary is convex.  On the other hand, there are counterexamples to the convexity of level sets for solutions to certain semi-linear PDEs \cite{MS, HNS}  and the mean curvature equation \cite{Wa}.
  
The parabolic version of this problem is far less developed.  The first major result is due to Borell \cite{Bo} who considered the heat equation on convex rings with zero initial data and proved space-time convexity of the superlevel sets.  Borell's result was extended to more general parabolic equations by Ishige-Salani \cite{Ish2, Ish3}, again assuming zero initial data.   For general quasiconcave initial data $u_0$, Ishige-Salani had shown that quasiconcavity of the superlevel sets is in general \emph{not} preserved \cite{Ish}.  Recently the authors gave counterexamples to preservation of quasiconcavity even under the additional assumption of subharmonicity of $u_0$ \cite{CW}, which was expected to be sufficient (cf. \cite{DK}).

We describe now Borell's result more precisely.
let $\Omega_0$ and $\Omega_1$ be bounded open convex bodies  in $\mathbb{R}^n$ with smooth boundaries  and $0 \in \Omega_1 \subset \subset \Omega_0$, and define $\Omega = \Omega_0 \setminus \Omega_1$.   Let $u$ solve
 \begin{equation} \label{borell}
 \left\{ \begin{array}{ll} \partial u/\partial t = {}  \Delta u, \quad &  (x,t) \in \Omega \times (0, \infty) \\
 u(x,0) = 0, \quad & x\in \Omega \\
 u(x,t) = 0, \quad &  (x,t) \in \partial \Omega_0 \times [0,\infty) \\
 u(x,t) = 1, \quad &  (x,t) \in \overline{\Omega}_1 \times [0,\infty). \end{array} \right.
 \end{equation}
Borell \cite{Bo} showed, using the language of probability and Brownian motion, that the level sets $\{ u = c \} \subset \Omega_0 \times [0,\infty)$ are convex hypersurfaces of $\mathbb{R}^{n+1}$. It is said that $u$ is \emph{space-time quasiconcave}. 
Our main result is an improvement from convexity to strong convexity.

\begin{theorem} \label{maintheorem} Let $u$ solve (\ref{borell}).  The level sets $\{ u = c \}$ for $c \in (0,1)$ are locally strongly convex hypersurfaces of $\mathbb{R}^{n+1}$.
\end{theorem}

We clarify now our terminology.  A smooth hypersurface $S$ in $\mathbb{R}^N$ is \emph{convex} if it is contained in the boundary of a convex body in $\mathbb{R}^N$.  It is 
 \emph{strongly convex} if
it can be represented locally around any $p\in S$ as the graph of a function $f$ with uniformly positive Hessian (its eigenvalues are bounded below by positive constants independent of $p$), and $S$ is \emph{locally strongly convex} if it is the union of strongly convex hypersurfaces.
A convex hypersurface $S$ is \emph{strictly convex} if it does not contain any line segment, a weaker condition than local strong convexity.  Note that we do not require $\Omega_0$ and $\Omega_1$ to have strongly or strictly convex boundaries.

  Borell  \cite{Bo2} introduced the notion of the \emph{parabolic convexity} of a set as follows.  We say that $E \subset \mathbb{R}^n \times [0,\infty)$ is parabolically convex if $X=(x,s), Y=(y,t) \in E$ implies that the parabolic segment
$$\lambda \mapsto P_{X, Y}(\lambda):= \left( (1-\lambda) x + \lambda y, ((1-\lambda) \sqrt{s} + \lambda \sqrt{t} )^2 \right) \quad \textrm{for } \lambda \in [0,1],$$
lies entirely in $E$.  
It was shown by Ishige-Salani \cite{Ish2} that solutions $u$ to (\ref{borell}), and for certain more general parabolic equations, have parabolically convex superlevel sets \cite{Ish2, Ish3}.   In the course of proving our main result, we will reprove the Ishige-Salani result for the heat equation.

 Our approach is different from the works above and applies the maximum principle to a parabolic version of  a two-point function of Rosay-Rudin \cite{RR}.   Namely, we will consider the function \begin{equation} \label{defnC1}
 \begin{split}
\mathcal{C}_p ((x,s),(y,t)) = {} & \frac{u(x,s)+u(y, t)}{2} - u \left( \frac{x+y}{2}, \left( \frac{s^{1/p}+t^{1/p}}{2} \right)^p \right) \\
= {} & u(x,s) - u \left( \frac{x+y}{2}, \left( \frac{s^{1/p}+t^{1/p}}{2} \right)^p \right) \\
\end{split}
\end{equation}
on
$$\Sigma = \{ ( (x,s), (y,t)) \in (\bar{\Omega} \times (0,\infty)) \times (\bar{\Omega} \times (0,\infty)) \ | \ u(x,s) = u(y,t)  \textrm{ and }  (x+y)/2 \in \bar{\Omega} \},$$
and for a constant $p \in [1,2]$.  { We first show that $\mathcal{C}_2 \leq 0$ on $\Sigma$.  Thus if $X,Y \in  \{ u = c\}$ then $P_{X, Y}(1/2)\in \{ u \ge c\}$ and it follows by an iterative argument that $P_{X, Y}(\lambda)\in \{ u \ge c\}$ for all $0\leq \lambda \leq 1$, in particular we obtain another proof of the parabolic convexity of superlevel sets of solutions $u$ to (\ref{borell}) (see Theorem \ref{thmIS}). 
We then show that { $\mathcal{C}_1 \le -c(|x_0-y|^2+|s_0-t|^2)$ for all $(y, t)$ in a neighborhood of any  $(x_0, s_0)$ in $\Sigma$ for some constant $c>0$, which in turn implies the strong convexity of the level sets of $u$ (see for example \cite[Section 3]{RR}).  

A brief outline of the proof is as follows.  Sections \ref{sectionRRL} and \ref{sectionmp} develop the parabolic version of the Rosay-Rudin two-point maximum principle.  A proof of the parabolic convexity of the superlevel sets using (\ref{defnC1}) is given in Section \ref{sectionpc}.  In Section \ref{sectionproof} we prove Theorem \ref{maintheorem} and finally in Section \ref{sectionremarks} we end with some remarks and open questions.

\emph{Note.} Shortly after this paper was posted on the arXiv preprint server,  an updated version of an article of Chen-Ma-Salani \cite{CMS} appeared, including a result related to Theorem \ref{maintheorem} proved via a constant rank theorem for the second fundamental form of the level surfaces (cf. Remark 5 in Section \ref{sectionremarks} below).}

\bigskip
\noindent
{\bf Acknowledgements.} \  We thank the referee for a careful reading of the paper and some helpful suggestions.

\section{A parabolic Rosay-Rudin Lemma} \label{sectionRRL}

 Consider  $u$ solving  (\ref{borell}).  We begin by proving a parabolic version of a lemma of Rosay-Rudin \cite[Lemma 1.3]{RR}.  Fix $T>0$ and interior points $(x_0, s_0)$ and $(y_0, t_0)$ in $\Omega \times (0,T]$ with $u(x_0, s_0) = u(y_0, t_0)$ and assume that $Du$, the spatial derivative of $u$, does not vanish at these points.  Let $L=(L_{ij}) \in O(n)$ satisfy 
$$L(Du(x_0, s_0)) = c Du(y_0, t_0), \quad \textrm{for } c = \frac{|Du(x_0, s_0)|}{|Du(y_0, t_0)|}.$$
We have:

\begin{lemma} \label{lemmaheat} Assume first that $s_0, t_0 \in (0,T)$.
There exists a smooth function $\psi(w, \tau) = O(|w|^3+ |\tau|^2)$ such that for all $(w,\tau) \in \mathbb{R}^n \times \mathbb{R}$ sufficiently close to the origin,
$$u(x_0 + w, s_0 + \tau) = u(y_0 + c L w + \chi(w, \tau) \xi + \psi (w, \tau) \xi, t_0+c^2\tau ), \quad \textrm{where  } \xi = \frac{Du(y_0, t_0)}{|Du(y_0, t_0)|},$$
for $\chi(w, \tau)$ defined by
$$\chi(w, \tau) = \frac{1}{|Du(y_0, t_0)|} ( u(x_0+w, s_0 + \tau) - u(y_0 + cLw, t_0  + c^2 \tau)),$$
which satisfies the heat equation $\frac{\partial \chi}{\partial \tau} = \Delta_w \chi.$

If $s_0$ or $t_0$ is equal to $T$, then the same holds with the additional restriction $\tau \le 0$.
\end{lemma}
\begin{proof}
Define a smooth real-valued function $G$ in a neighborhood of zero in $(\mathbb{R}^n \times \mathbb{R}) \times \mathbb{R}$ by
$$G((w, \tau), \psi) = u(y_0 + cL w + \chi(w, \tau) \xi + \psi \xi, t_0 + c^2 \tau) - u(x_0 + w, s_0 +\tau),$$
which satisfies $G((0,0), 0) = 0$.  Compute
$$\frac{\partial G}{\partial \psi}((0,0),0) = \sum_i D_i u (y_0, t_0) \, \xi_i = |Du(y_0, t_0)| >0.$$
Hence by the Implicit Function Theorem, there exists a smooth $\psi = \psi (w, \tau)$ satisfying 
$$G((w, \tau), \psi(w, \tau))=0,$$
for $w, \tau$ close to zero.

It remains to show that $\psi(w, \tau) = O(|w|^3+|\tau|^2)$.  First compute at the origin
\[
\begin{split}
0 = \frac{\partial G}{\partial w_j} = {} & u_i (y_0,t_0) \left( cL_{ij} + \frac{1}{|Du(y_0, t_0)|} (u_j(x_0,s_0) - c u_k (y_0, t_0) L_{kj} ) \xi_i + \frac{\partial \psi}{\partial w_j} \xi_i \right) \\ {} & - u_j(x_0, s_0) \\
= {} & |Du(y_0, t_0)| \frac{\partial \psi}{\partial w_j}(0,0),
\end{split}
\]
which implies that $\frac{\partial \psi}{\partial w_j}(0,0)=0$.  Next, 
\[
\begin{split}
0 =  \frac{\partial^2 G}{\partial w_j w_p} = {} & u_{ik} (y_0, t_0) c^2 L_{kp} L_{ij}   + u_i(y_0, t_0)\frac{ \left( u_{jp} (x_0, s_0) - c^2 u_{k\ell}(y_0, t_0) L_{kj} L_{\ell p} \right) \xi_i }{|Du(y_0, t_0)|}  \\ {} & + u_i (y_0, t_0) \frac{\partial^2 \psi}{\partial w_p \partial w_j}(0,0) \xi_i - u_{jp} (x_0, s_0) \\ 
= {} & | Du(y_0, t_0)| \frac{\partial^2 \psi}{\partial w_p \partial w_j}(0,0),
\end{split}
\]
so that $\frac{\partial^2 \psi}{\partial w_p \partial w_j}(0,0)=0$.  Finally,
\[
\begin{split}
0 = \frac{\partial G}{\partial \tau} = {} & \xi_i u_i(y_0, t_0) \frac{\partial \chi}{\partial \tau}   + \xi_i u_i (y_0, t_0) \frac{\partial \psi}{\partial \tau} + c^2 u_t(y_0, t_0) - u_t(x_0, s_0) \\
= {} & |Du(y_0, t_0)| \frac{1}{|Du(y_0, t_0)|} (u_t(x_0, s_0) - c^2 u_t(y_0, t_0)) + |Du(y_0, t_0)|  \frac{\partial \psi}{\partial \tau} \\ {} & + c^2 u_t(y_0, t_0) - u_t(x_0, s_0) \\
= {} &  |Du(y_0, t_0)|  \frac{\partial \psi}{\partial \tau},
\end{split}
\]
giving $\frac{\partial \psi}{\partial \tau}(0,0)=0,$ as required.
\end{proof}

We end this section with another technical lemma.  Using the notation of Lemma \ref{lemmaheat}, we write
$$(x,s) = (x_0+w, s_0+ \tau)$$ and $$(y,t) = (y_0 + cLw + \chi(w,\tau)\xi + \psi(w,\tau) \xi, t_0+c^2\tau).$$
Then, evaluating at $(y_0, t_0)$,
\begin{equation}\label{heaty}
\begin{split}
\frac{\partial}{\partial \tau} y_i = \frac{\partial \chi}{\partial \tau} \xi_i = (\Delta_w \chi) \xi_i = \Delta_w y_i. \end{split}
\end{equation}
We make use of this in the following lemma.

\begin{lemma} \label{lemmav}
With the notation above, if $v$ is a solution to the heat equation then 
$$\left( \Delta_w - \partial_{\tau} \right) v(x,s) =0, \quad \textrm{and } \left( \Delta_w - \partial_{\tau} \right) v(y,t) =0,$$
when evaluated at $(w, \tau)=(0,0)$.
\end{lemma}
\begin{proof}
The first equation is immediate.  For the second, computing at $(0,0)$,
\[
\begin{split} 
\left( \Delta_w - \partial_{\tau} \right) v(y,t) 
= {} & \sum_j v_{ik}(y,t) \frac{\partial y_k}{\partial w_j} \frac{\partial y_i}{\partial w_j} + \sum_j v_i(y,t) \frac{\partial^2 y_i}{\partial w_j^2} -  v_i(y,t) \frac{\partial y_i}{\partial \tau} \\ {} & - c^2 v_t(y,t) \\
= {} & \sum_j c^2  v_{ik}(y,t) L_{kj} L_{ij} - c^2 v_t(y,t) \\
= {} &  c^2 \Delta v(y,t) - c^2 v_t(y,t) =0.
\end{split}
\]
completing the proof.
\end{proof}

\section{Maximum principle for a two-point function} \label{sectionmp}
 Recall the following family of functions in \eqref{defnC1}:
\begin{equation}
 \begin{split}
\mathcal{C}_p ((x,s),(y,t))= {} & u(x,s) - u \left( \frac{x+y}{2}, \left( \frac{s^{1/p}+t^{1/p}}{2} \right)^p \right) \\
\end{split}
\end{equation}
In this section we prove a parabolic maximum principle for a slight modification of these functions analogous to the function introduced by Rosay-Rudin \cite{RR}.   We begin by recalling some basic properties of the solution $u$ to (\ref{borell}).

\begin{proposition} \label{basic} We have
\begin{enumerate}
\item[(i)] $0 < u < 1$ on $\Omega \times (0,\infty)$.
\item[(ii)] If $x \in \ov{\Omega}_1$ and $w \in \Omega$ then $(w-x) \cdot Du(w, t) < 0$ for $t\in (0,\infty)$.
\item[(iii)] $\Delta u > 0$ on $\Omega \times (0,\infty)$.
\end{enumerate}
\end{proposition}
\begin{proof}
This is well-known, as a consequence of the maximum principle (see \cite{Bo,DK} for example).
\end{proof}

Fix $p \in [1,2]$ and $T \in (0,\infty)$.   Let $h_1, \ldots, h_N$ be arbitrary solutions of the heat equation on $\Omega \times (0,T]$.  In the later sections we will in fact only make use of $p=1$ or $p=2$, and we will take $N=1$. We also fix a small constant $\delta>0$.
We consider the quantity
 \begin{equation} \label{defnC}
\begin{split}
 \mathcal{Q} ((x,s),(y,t)) :=& \mathcal{C}_p ((x,s),(y,t)) + \sum_{i=1}^N ( h_i(x,s) - h_i(y,t))^2 - \delta s\\\end{split}
\end{equation}
on
$$\Sigma = \{ ( (x,s), (y,t)) \in (\bar{\Omega} \times (0,\infty)) \times (\bar{\Omega} \times (0,\infty)) \ | \ u(x,s) = u(y,t)  \textrm{ and }  (x+y)/2 \in \bar{\Omega} \}.$$
We say that $((x,s), (y,t)) \in \Sigma$ is an \emph{interior point} of $\Sigma$ if  $x, y, (x+y)/2 \in \Omega$. Note that $s$ or $t$ are allowed to be equal to $T$.
The result of this section is the following maximum principle, which is a parabolic analogue of \cite[Theorem 4.3]{RR}.

\begin{proposition} \label{propM}
$\mathcal{Q}$ does not attain a maximum at an interior point of $\Sigma$. 
\end{proposition}

\begin{proof}  First we assume that $n$ is even.  
Suppose for a contradiction that $\mathcal{C}$ achieves a maximum at some interior point of $\Sigma$, which we will call $((x_0, s_0), (y_0, t_0))$.  We will rule this out.

We apply Lemma \ref{lemmaheat} and use the notation there.  Note that by part (ii) of Proposition \ref{basic}, $Du$ does not vanish at $(x_0, s_0)$ or $(y_0, t_0)$.  For sufficiently small $\tau \in \mathbb{R}$ and $w \in \mathbb{R}^n$, define
$$(x,s) = (x_0+w, s_0+ \tau)$$ and $$(y,t) = (y_0 + cLw + \chi(w,\tau)\xi + \psi(w,\tau) \xi, t_0+c^2\tau)$$ and consider 
$$ F(w, \tau) = \mathcal{Q}((x,s),(y,t)).$$
Note that if one of $s_0, t_0$ is equal to $T$ then we must restrict to $\tau$ to be nonpositive.
Write $$Z = \left(\frac{x+y}{2}, \left( \frac{s^{1/p}+t^{1/p}}{2} \right)^p\right).$$  Then
\[
\begin{split}
\sum_j \frac{\partial^2}{\partial w_j^2} u\left( Z \right) = {} & \sum_j  u_{k\ell} \left( Z \right) \frac{1}{4} (\delta_{kj} + c L_{kj})(\delta_{\ell j} + c L_{\ell j})  + u_k\left( Z \right) \frac{1}{2} \Delta_w y_k \\
= {} & \frac{1+c^2}{4} \Delta u \left( Z \right)  + \frac{c}{2} L_{k\ell} u_{k\ell} \left( Z \right)  + u_k\left( Z \right) \frac{1}{2} \Delta_w y_k. \\
\end{split}
\]

We make an appropriate choice of $L$ following \cite[Lemma 4.1(a)]{RR}, recalling our assumption that $n$ is even.  Namely, after making an orthonormal change of coordinates, we may assume, without loss of generality that $Du(x_0, s_0)/|Du(x_0,s_0)|$ is $e_1$, and 
$$Du(y_0,t_0)/|Du(y_0, t_0)| = \cos \theta \, e_1 + \sin \theta \, e_2,$$
for some $\theta \in [0,2\pi)$. Here we are writing $e_1 = (1,0,\ldots 0)$ and $e_2 = (0,1,0, \ldots)$ etc for the standard unit basis vectors in $\mathbb{R}^n$.   Observe that
\begin{equation} \label{ccos}
\cos \theta = \frac{Du(x_0, s_0)}{|Du(x_0, s_0)|} \cdot \frac{Du(y_0, t_0)}{|Du(y_0, t_0)|}, \quad c=\frac{|Du(x_0, s_0)|}{|Du(y_0, t_0)|}.
\end{equation}

Then define the isometry $L$ by 
$$L(e_i) = \left\{ \begin{array}{ll} \cos \theta \, e_i + \sin \theta \, e_{i+1}, & \quad \textrm{for } i=1,3, \ldots, n-1 \\
- \sin \theta \, e_{i-1} + \cos \theta \, e_{i}, & \quad \textrm{for } i=2, 4, \ldots, n.\end{array} \right.$$
In terms of entries of the matrix $(L_{ij})$, this means that $L_{kk} = \cos \theta$ for $k=1, \ldots n$ and for $\alpha = 1, 2, \ldots, n/2$, we have
$$L_{2\alpha-1, 2\alpha}= - \sin \theta, \quad L_{2\alpha, 2\alpha -1} = \sin \theta,$$
with all other entries zero.
Then for any point,
\begin{equation*} \label{Lki}
\begin{split}
\sum_{i,k} L_{ki} u_{ki} = 
  (\cos \theta) \Delta u.
\end{split}
\end{equation*}
Hence
\[ \begin{split}
\sum_j \frac{\partial^2}{\partial w_j^2} u\left( Z \right) = {}  & \frac{1+c^2 + 2c \cos \theta}{4} \Delta u \left( Z \right)  + u_k\left( Z \right) \frac{1}{2} \Delta_w y_k. \\
\end{split}
\]
Compute
\[ 
\begin{split}
\frac{\partial}{\partial \tau} u\left( Z \right) =  {} & \frac{1}{2} \left(  \left( \frac{s_0^{1/p}+t_0^{1/p}}{2} \right)^{p-1} (s_0^{\frac{1}{p}-1} + c^2 t_0^{\frac{1}{p}-1} ) \right)u_t \left( Z \right) + u_k \left( Z \right) \frac{1}{2} \frac{\partial y_k}{\partial \tau} \\
= {} &  \frac{1}{2} \left(  \left( \frac{1+ (t_0/s_0)^{1/p}}{2} \right)^{p-1} + c^2 \left( \frac{1+ (s_0/t_0)^{1/p}}{2} \right)^{p-1} \right)u_t \left( Z \right) +  u_k \left( Z \right) \frac{1}{2} \frac{\partial y_k}{\partial \tau} \\ 
\ge {} & \frac{1}{4} \left( 1+c^2 + (t_0/s_0)^{(p-1)/p} + c^2 (s_0/t_0)^{(p-1)/p} \right) u_t \left( Z \right) +  u_k \left( Z \right) \frac{1}{2} \frac{\partial y_k}{\partial \tau}, \\ 
\end{split}
\]
where for the last line we used $u_t(Z)= \Delta u(Z)\ge 0$ from Proposition \ref{basic} and the concavity of the map $x \mapsto x^{p-1}$.  Note that the inequality is an equality in the cases $p=1$ and $p=2$.

Hence, using (\ref{heaty}),
$$\left( \Delta_w - \frac{\partial}{\partial \tau} \right) u\left( Z \right) \le \frac{1}{4} \left( 2c \cos \theta - (t_0/s_0)^{(p-1)/p} - c^2 (s_0/t_0)^{(p-1)/p} \right) \Delta u \left( Z \right).$$
Putting this together, we obtain at $(w, \tau)=(0,0)$, using Lemma \ref{lemmav} and (\ref{ccos}),
\[
\begin{split}
\left( \Delta_w - \frac{\partial}{\partial \tau} \right) F \geq {} &  \frac{1}{4} \left( - 2c \cos \theta + (t_0/s_0)^{(p-1)/p} + c^2 (s_0/t_0)^{(p-1)/p} \right) \Delta u \left( Z \right)  \\
{} & + 2 \sum_{j=1}^n \sum_{i=1}^N  \left( \frac{\partial}{\partial w_j} (h_i(x,s) - h_i(y,t)) \right)^2 + \delta \\
\ge {} &  \frac{1}{4 |Du(y_0, t_0)|^2} \left|  \left( \frac{s_0}{t_0} \right)^{\frac{p-1}{2p}} Du(x_0, s_0) - \left( \frac{t_0}{s_0} \right)^{\frac{p-1}{2p}} Du(y_0, t_0)\right|^2  \Delta u \left(Z \right) + \delta \\ > {} & 0.
\end{split}
\]
This contradicts the fact that $F$ attains a maximum at this point.

Finally, we deal with the case when $n$ is odd, making modifications analogous to those in \cite{RR}.  Namely, define $L$ to be an isometry of $\mathbb{R}^{n+1}$ satisfying $L(Du(x_0, s_0),0)=(c(Du)(y_0, t_0), 0)$ and in Lemma \ref{lemmaheat} we consider $w\in \mathbb{R}^{n+1}$.  Writing $\pi$ for the projection $(w_1, \ldots, w_{n+1}) \mapsto (w_1, \ldots, w_n)$ the statement of Lemma \ref{lemmaheat} becomes
$$u(x_0+\pi(w), s_0 + \tau) = u(y_0 + c \pi(Lw) + \chi(w,\tau)\xi + \psi (w, \tau)\xi, t_0+c^2\tau),$$
for the same $\xi$ and with $$\chi(w, \tau) = \frac{1}{|Du(y_0, t_0)|} \left( u(x_0 + \pi(w), s_0+\tau) - u(y_0+c\pi(Lw), t_0+c^2 \tau) \right),$$
which satisfies the heat equation in a neighborhood of the origin in $\mathbb{R}^{n+1} \times \mathbb{R}$.  The rest of the proof then goes through with the obvious changes.
\end{proof}

\section{Parabolic convexity} \label{sectionpc}

In this section we give a proof of a result of Ishige-Salani \cite{Ish3} that the superlevel sets of   $u$ solving (\ref{borell}) are parabolically convex.  Our proof is somewhat different, and uses the following two point function from \eqref{defnC1}:
$$ \mathcal{C}_2((x,s), (y,t)) = u(x,s) - u\left( \frac{x+y}{2}, \left( \frac{\sqrt{s}+\sqrt{t}}{2} \right)^2 \right)$$
on
$$\Sigma = \{ ( (x,s), (y,t)) \in (\bar{\Omega} \times (0,\infty)) \times (\bar{\Omega} \times (0,\infty)) \ | \ u(x,s) = u(y,t)  \textrm{ and }  (x+y)/2 \in \bar{\Omega} \}.$$

\begin{theorem} \label{thmIS}
We have $ \mathcal{C}_2 \le 0$ on $\Sigma$.  Equivalently, the superlevel sets of $u$ are parabolically convex.
\end{theorem}

\begin{proof} Fix $T \in (0,\infty)$ and a small $\delta>0$ and consider the quantity
$$ \mathcal{Q} ((x,s),(y,t)) = \mathcal{C}_2 ((x,s),(y,t))  -\delta s.$$
We will prove $ \mathcal{Q} \le 0$ on $\Sigma_T=\Sigma \cap \{ 0 \le s,t \le T\}$ and the result will follow from letting $\delta \rightarrow 0$ and $T \rightarrow \infty$.

From Proposition \ref{propM}, we only need to consider the case when $(x,s)$ or $(y,t)$ are boundary points of $\Sigma_T$.  If $s$ and $t$ are both positive, there are four cases:
\begin{enumerate}
\item If  $x$ or $y$  lie on $\partial \Omega_1$, by convexity of $\Omega_1$ and the definition of $\Sigma_T$, the only possibility is that $x$, $y$ and $(x+y)/2$ all lie on $\partial \Omega_1$ and by the boundary conditions for $u$ we have $ \mathcal{Q}  \le 0$.
\item If $x$ or $y$ lie on $\partial \Omega_0$ then $u(x,s)=0=u(y,t)$ and $ \mathcal{Q} \le 0$.
\item If $(x+y)/2 \in \partial \Omega_0$ then we must have $x$ and $y$ in $\partial \Omega_0$ by convexity of $\Omega_0$ and we are in case (2).
\item If $(x+y)/2 \in \partial \Omega_1$ then $u((x+y)/2, (\sqrt{s}/2+\sqrt{t}/2)^2)=1$ and $ \mathcal{Q} \le 0$.
\end{enumerate}

It remains to deal with the case when $s$ or $t$ (or both) tend to zero.  A difficulty here is that $u$ is discontinuous at $t=0$ at the boundary of $\Omega_1$.   Assume we have a sequence of points $X_i = (x_i, s_i)$ and $Y_i = (y_i, t_i)$ in $\Sigma_T$ for which $ \mathcal{Q}(X_i, Y_i) \ge \eta$ for a positive constant $\eta>0$.  Define $Z_i = (z_i, r_i)$, where $z_i = (x_i+y_i)/2$ and $r_i = (\sqrt{s_i}/2+\sqrt{t_i}/2)^2$.
Assume  that $$X_i \rightarrow X = (x, s), \ Y_i \rightarrow Y= (y,t), \  Z_i \rightarrow Z = (z, r)$$
$$\textrm{for }  z=(x+y)/2 \textrm{ and }r=(\sqrt{s}/2+\sqrt{t}/2)^2).$$
We also assume, without loss of generality, that $s \le t$.  There are two cases.

\medskip
\noindent
{\bf (i) The case when $t>0$ and $s=0$}. \  We must have $x \in \partial \Omega_1$ since otherwise this would contradict the inequality $ \mathcal{Q}(X_i, Y_i) \ge \eta$.  By the same reasoning as in (1)-(4) above, we may also assume that $y$ and $z=(x+y)/2$ lie in $\Omega$.  We have the following lemma:

\begin{lemma}\label{parabolic-convexity-along-parabolas}  Suppose that $0 \le s<t$ and $x \in  \partial \Omega_1$, $y \in \Omega$. Then
\begin{equation} \label{map}
\frac{d}{d\lambda} u((1-\lambda) x + \lambda y, ((1-\lambda)\sqrt{s}+\lambda \sqrt{t})^2) \le 0,
\end{equation}
whenever $\lambda \in (0,1]$ and $(1-\lambda)x + \lambda y \in \Omega$.
\end{lemma}
\begin{proof} We recall a differential inequality of Borell \cite[(2.1)]{Bo}.  If $x \in \partial \Omega_1$ and $w\in \Omega$, we have
\begin{equation} \label{diffe}
(w-x) \cdot D u(w,t) + 2t u_t (w,t) \le 0, \textrm{for } t>0,
\end{equation}
where  $Du$ is the spatial derivative of $u$.  In fact, Borell   used probablistic methods to derive a sharper inequality, but for our purposes, (\ref{diffe}) suffices.
For convenience of the reader, we include here the brief proof of (\ref{diffe}), following \cite[Lemma 4.4]{Ish2}.  Assume without loss of generality that $x$ is the origin.  Consider for 
$\sigma \in [0,1]$ the quantity $$W(\zeta, t)=u(\sigma\zeta, \sigma^2 t) - u(\zeta, t)$$ on the set where $\sigma \zeta, \zeta \in \Omega$.  On the boundary of its domain, $W$ is nonnegative, and $W$ vanishes at $t=0$.  Since $W$ solves the heat equation, the weak maximum principle implies that $W \ge 0$.  Differentiating with respect to $\sigma$  and evaluating at $\sigma=1$ gives (\ref{diffe}).

We now prove the lemma. Writing $w= (1-\lambda)x + \lambda y$ and $\rho=(1-\lambda)\sqrt{s}+\lambda \sqrt{t}$ we have
\[
\begin{split}
\frac{d}{d\lambda} u(w,\rho^2) = {} & (y-x) \cdot Du(w, \rho^2 ) + 2\rho (\sqrt{t}-\sqrt{s}) u_t (w, \rho^2)  \\
= {} & \frac{1}{\lambda} ( w-x) \cdot Du(w, \rho^2 ) + 
\frac{2\rho^2}{\lambda} u_t (w,\rho^2)  - \frac{2\rho \sqrt{s}}{\lambda} u_t (w,\rho^2) \le 0.\\
\end{split}
\]
Indeed (\ref{diffe}) implies that sum of the first two terms is nonpositive, and the last term is nonpositive since $u_t>0$.
 \end{proof}

The points $X, Z$ and $Y$ have coordinates $((1-\lambda)x+\lambda y, ((1-\lambda) \sqrt{s} +\lambda \sqrt{t})^2)$ with $\lambda =0, 1/2$ and $1$ respectively.  Since $x \in \partial \Omega_1$ and $y,z \in \Omega$ it follows that the line segment $(1-\lambda )x + \lambda y$ for $\lambda \in [1/2, 1]$, which goes from $z$ to $y$, lies completely in $\Omega$.  Lemma \ref{parabolic-convexity-along-parabolas} implies that $u(Y) \le u(Z)$ and by the continuity of $u$ at $Y$ and $Z$ we see that $u(X_i)=u(Y_i) \le u(Z_i) + \eta/2$ for $i$ sufficiently large, contradicting $ \mathcal{Q}(X_i, Y_i) \ge \eta$.

\medskip
\noindent
{\bf (ii) The case when  $s$ and $t$ are both zero}.  \ Our line of reasoning in this case is analogous to the probablistic argument of \cite[Section 3]{Bo}. The points $x, y$ and $z$ must lie on the boundary $\partial \Omega_1$.  Now for each $i$, we can find an affine transformation $T_i :\mathbb{R}^n \rightarrow \mathbb{R}^n$ such that the function $$u_i(w, t):= u(T_i^{-1}w, t)$$ still solves the boundary value problem \eqref{borell}, but on the transformed domain $T_i(\Omega)$ in the coordinates $w_1,..,w_n$, and:
\begin{enumerate}

\item [(a)] $T_i(\Omega_1)$ is tangent to the hyperplane $w_1=0$ at the origin, and lies in the half space $w_1 \leq 0$; 

\item [(b)]  $T_i(z_i)$ lies on the $w_1$ axis.

\end{enumerate}

 Let $v(w_1,..,w_n,t)$ be defined on $\mathbb{R}^n \times [0, \infty)$ as being identically $1$ when $w_1\leq 0$ and otherwise  given by the solution of the heat equation on the half space $w_1 \ge 0$ with initial condition $v=0$ when $w_1>0$, and boundary condition $v=1$ on $w_1=0$.  We can write down $v$ explicitly as
 \begin{equation} \label{v}
v(w, t) = \Psi \left( \frac{w_1^2}{t} \right), \quad \Psi (\lambda) = \int_0^{1/\lambda} (4\pi \sigma^3)^{-1/2} \exp(-1/(4\sigma))d\sigma.
\end{equation}
In particular, note that the level sets of $v$ are given by $t=cw_1^2$ for $c>0$ from which it is straightforward to show that the superlevel sets of $v$ are parabolically convex.  Moreover, the maximum principle implies that $u_i(w, t) \le v(w, t)$ on $T_i(\Omega) \cap \{ w_1 \ge 0 \}$.  We have the following claim.

\medskip
\noindent
{\bf Claim.} \ For compact subsets $K \subset ( \{ w_1 \ge 0 \} \times (0,T] )$, and any positive sequence $a_i \to 0$,
$$u_i(a_i w, a_i^2 t) \rightarrow v(w,t), \quad \textrm{as } i \rightarrow \infty,$$
uniformly for $(w,t) \in K$. 

\medskip
\begin{proof}[Proof of Claim] This follows from the fact that the function $\tilde{u}_{i}(w,t) = u_i(a_i w, a_i^2 t)$ solves the heat equation on $(1/a_i )T_i(\Omega)$, and as $i\to  \infty$ the boundary conditions of $\tilde{u}_{i}$ approach those of $v$.  To make this more precise, assume $K$ lies in $B_R \cap \{ w_1 \ge 0\} \times [\delta, T]$ for $\delta>0$, where $B_R$ is a ball in $\mathbb{R}^n$ of radius $R>0$ centered at the origin.  Fix $\varepsilon>0$.

For a small $\beta>0$, we define $v_{\beta}$ to be the translate of $v$ in the negative $w_1$ direction by the amount $\beta$, namely $v_{\beta}(w, t) = \Psi((w_1+\beta)^2/t)$.  Pick $\beta$ sufficiently small  so that on the compact set $K$, 
\[
\begin{split}
|v-v_{\beta}| = {} &   \int^{t/w_1^2}_{t/(w_1+\beta)^2} (4\pi \sigma^3)^{-1/2} \exp(-1/(4\sigma))d\sigma 
< \ve.
\end{split}
\]
The function $v_{\beta}$ solves the heat equation on the set $\{ w_1 > -\beta \}$ with zero initial data and boundary condition $v_{\beta}=1$ on $\{w_1 = -\beta\}$.

Next for $S>R>0$, define a function $\varphi_S(w,t)$ to be a solution to the heat equation on $\{w_1 > -\beta \} \cap B_S$ with zero initial data and boundary condition given by
$$\varphi_S(w, t) = \left\{ \begin{array}{ll} 0, \quad & \textrm{ for  }w \in B_S \cap \{w_1 = - \beta \} \\ 1, \quad & \textrm{ for }w\in \partial(B_S) \cap \{w_1 > - \beta\} \end{array} \right.$$
We choose $S$ sufficiently large so that $\varphi_S \le \ve$ on $K$.

Now choose $i$ sufficiently large, depending on $S$, so that $((1/a_i) T_i(\Omega)) \cap B_S$ lies entirely in the set $\{ w_1 > -\beta \}$, or in other words $B_S \cap \{ w_1 \le - \beta \}$ is contained in $(1/a_i)T_i(\Omega_1)$.  We may also assume without loss of generality that the boundary of $T_i(\Omega_0)$ lies outside $B_S$.  Now the function $\tilde{u}_i+\varphi_S -v_{\beta}$ solves the heat equation on $((1/a_i) T_i(\Omega)) \cap B_S$ for $t\in [0, a_i^{-2} T]$ with zero initial data, and strictly positive boundary condition by construction.  Indeed on the part of the boundary which coincides with the boundary of $T_i(\Omega_1)$ we have $\tilde{u}_i= 1 \ge v_{\beta}$, and on the rest of the boundary we have $\varphi_S =1 \ge v_{\beta}$.
Hence $\tilde{u}_i + \varphi_S - v_{\beta} \ge 0$, and hence on $K$ we have 
$$\tilde{u}_i \ge v - 2\varepsilon.$$
Since we  have $\tilde{u}_i \le v$ by the maximum principle this completes the proof of the claim.
\end{proof}

Recall that we have a sequence $X_i=(x_i, s_i)$, $Y_i=(y_i, t_i)$ with $\mathcal{C}(X_i, Y_i) \ge \eta>0$.  Writing $\tilde{X}_i = (\tilde{x}_i, s_i)=(T_i(x_i), s_i)$ and similarly for $\tilde{Y}_i$ and $\tilde{Z}_i$ we have

\begin{equation} \label{eqc}
u_i(\widetilde{Z}_i) +\eta \le u_i(\widetilde{X}_i) \le \min (v(\widetilde{X}_i), v(\widetilde{Y}_i)) \le v(\widetilde{Z}_i).
\end{equation}
Here the second inequality follows from $u_i(\widetilde{X}_i)=u_i(\widetilde{Y}_i)$ and $u_i \le v$, while the third inequality follows from the parabolic quasiconcavity of $v$.

Now  $v(\widetilde{Z}_i)=v(\tilde{z}_i, r_i)  \ge \eta>0$ and it follows that $\rho_i := (w_1(\tilde{z}_i))^2/r_i \le C$ for a uniform constant $C$, since $\Psi(\lambda) \rightarrow 0$ as $\lambda \rightarrow \infty$.  Here we are writing $w_1(\tilde{z}_i)$ for the $w_1$ coordinate of $\tilde{z}_i=T_i(z_i)$.  After passing to a subsequence, we may assume that $\rho_i \rightarrow \rho <\infty$.

Then using the above, and recalling the properties of the transformation $T_i$ we have
\begin{equation} \label{equZn>1}
u_i(\widetilde{Z}_i)=u_i (w_1(\tilde{z}_i), 0,..,0, r_i) = u_i(\sqrt{\rho_i} \sqrt{r_i}, 0,...,0, r_i) \rightarrow v(\sqrt{\rho}, 0,...,0,1),
\end{equation} 
as $i \rightarrow \infty$, using the Claim with $a_i = \sqrt{r_i}$.  But 
\begin{equation} \label{eqvZ}
v(\widetilde{Z}_i)=v(\sqrt{\rho_i} \sqrt{r_i},0,...,0, r_i) = v(\sqrt{\rho_i}, 0,...,0, 1) \rightarrow v(\sqrt{\rho}, 0,...,0,1),
\end{equation}
as $i \rightarrow \infty$
which contradicts (\ref{eqc}).  
\end{proof}

\section{Proof of Theorem \ref{maintheorem}} \label{sectionproof}

In this section, we give the proof of Theorem \ref{maintheorem}.  
Fix $T \in (0,\infty)$ and $0<\mu<1$.  Let $S_{\mu}=\{ (x, t) \in \Omega \times (0,T] \  | \ u(x, t)=\mu\}$, which is convex from Borell's result \cite{Bo} (or, from Theorem \ref{thmIS}).  We wish to show strong convexity of $S_{\mu}$ on compact subsets.   We will do this using the two point function 
\begin{equation}
 \begin{split}
\mathcal{C}_1 ((x,s),(y,t))= {} & u(x,s) - u \left( \frac{x+y}{2}, \left( \frac{s+t}{2} \right) \right) \\
\end{split}
\end{equation}
from \eqref{defnC1}.  For $\alpha, \beta$  with $0<\alpha < \mu<\beta< 1$ we define the space-time region 
$$\Xi = \{ (x,t) \in \Omega \times (0, T] \ | \ \alpha \le u(x,t) \le \beta \},$$
bounded by $\{t =T\}$ and ``inner boundary'' $S_{\beta}$ and ``outer boundary''   $S_{\alpha}$, which are defined in the same way as $S_{\mu}$.   The following lemma is the key result of this section.

\begin{lemma} \label{keylemma} Fix $(x_0, t_0)$ in $S_{\mu}$ and a unit vector $V=(V_1, \ldots, V_{n+1}) \in \mathbb{R}^{n+1}$ with $V_{n+1} \ge 0$.  Then there exist $\alpha$ and $\beta$ with $0<\alpha < \mu<\beta< 1$ and a smooth function $h$ on $\Xi$ satisfying the heat equation in the interior of $\Xi$ such that:
\begin{enumerate}
\item[(i)] The function
\begin{equation} \label{defnC2}
\mathcal{Q} ((x,s),(y,t)) =  \mathcal{C}_1 ((x,s),(y,t))
+  ( h(x,s) - h(y,t))^2,
\end{equation}
defined on $$\Sigma' = \{ ((x,s), (y,t)) \in \Xi \ | \ u(x,s)=u(y,t), \ ((x+y)/2, (s+t)/2) \in \Xi \},$$
 is nonpositive.
\item[(ii)] We have
$$ \nabla_Vh(x_0, t_0) \neq 0.$$
\end{enumerate}
\end{lemma}

Here we are using $\nabla_Vh$ to denote the space-time directional derivative $$\sum_{i=1}^n V_i D_i h+ V_{n+1} \partial_t h.$$
Given the lemma we can complete the proof of the main theorem.

\begin{proof}[Proof of Theorem \ref{maintheorem}]
This is an almost immediate consequence of Lemma \ref{keylemma}. 
Fix $(x_0, t_0)$ in $S_{\mu}$.  
By compactness of the unit sphere in $\mathbb{R}^{n+1}$ we obtain in a neighborhood of $(x_0, t_0)$,
$$  \mathcal{C}_1 ((x,s),(y,t))+ c(|x-y|^2 + |s-t|^2) \le 0,$$
for a uniform constant $c>0$.  It follows that any compact subset of $S_{\mu}$ is strongly convex (see for example \cite[Section 3]{RR}) as required.  
\end{proof}

It remains then to prove the lemma.

\begin{proof}[Proof of Lemma \ref{keylemma}]

Fix  $(x_0, t_0)$ and a unit vector $V$ as in the statement of the lemma.  We first make the following claim.

\bigskip
\pagebreak[3]
\noindent
 {\bf Claim.} \  There exists $0<\alpha<\mu<\beta<1$, a strongly convex open set $E_{\alpha, \beta}$ of $S_{\alpha} \bigcap \{ t > t_0/2 \}$  and a smooth compactly supported  $f: E_{\alpha, \beta} \rightarrow (0,\beta-\alpha)$
  such that 
\begin{enumerate}
\item[(a)] There exists a unique solution $h(x,t)$ to the heat equation $\partial h/\partial t = {}  \Delta h$ on $\Xi$ with boundary conditions
 \begin{equation}\label{hequation}
 \left\{ \begin{array}{ll} 
 h(x,t) = 0, \quad &  (x,t) \in (S_{\alpha}\setminus E_{\alpha, \beta}) \cup S_{\beta} \\
 h(x,t) = f(x,t), \quad &  (x,t) \in E_{\alpha, \beta} \\
\end{array} \right.
 \end{equation}
and initial data $h(x,0)=0$.
\item[(b)] $\nabla_V h(x_0, t_0)\neq  0$.
\end{enumerate}

\begin{proof}[Proof of Claim]   We first prove part (a).  In particular we show the existence of a solution $h(x,t)$ as in (a) given any $0<\alpha<\mu<\beta<1$, $E_{\alpha, \beta}$ and $f$ as in the hypothesis of the claim.  To deal with the fact that the boundary is changing in time,
we consider $\Xi_0 := \Xi \cap \{ t \ge t_0/2 \}$ which is diffeomorphic to the cylinder $\ov{\Omega} \times [t_0/2, T]$.  Indeed, there is a diffeomorphism $\Psi_{\alpha, \beta}:\ov\Xi_0\to  \ov{\Omega} \times [t_0/2, T]$ satisfying
\begin{enumerate}
\item [(i)] $\Psi_{\alpha, \beta}$ is the identity in the $t$ factor and maps each time slice $\ov{\Xi_0} \cap \{ t=t' \}$ diffeomorphically to $\ov{\Omega} \setminus \Omega_1$, and is a diffeomorphism of the boundary $S_{\alpha} \cap \{ t=t'\}$ to $\partial \Omega_0$ and $S_{\beta} \cap \{t=t'\}$ to $\partial \Omega_1$.
\item [(ii)] $\Psi^{-1}_{\alpha, \beta}$ converges smoothly uniformly to the identity on $\ov{\Omega} \times [t_0/2, T]$ as $(\alpha, \beta) \to (0, 1)$. 
\end{enumerate}

We write $\Psi$ for $\Psi_{\alpha, \beta}$. Define $F: \Psi(E_{\alpha, \beta}) \rightarrow (0,\beta-\alpha)$ by $F(w,t) = f( \Psi^{-1}(w, t))$.  Now let $H(x,t)$ be the solution of the parabolic equation
\begin{equation} \label{mode}
H_t=
 (\Psi^k_i \Psi^j_i ) H_{jk} + (\Psi^j_{ii} - \Psi_t^j) H_j
 \end{equation}
 on the space-time cylinder $\Omega \times (t_0/2,T]$, with zero initial data $H=0$ at $t=t_0/2$ and boundary conditions 
 \begin{equation}\label{Hequation}
 \left\{ \begin{array}{ll}  
 H(w,t) = F(w,t), \quad &  (w,t) \in \Psi(E_{\alpha, \beta}) \\
 H(w,t)= 0, \quad & \textrm{otherwise}.
\end{array} \right.
 \end{equation}
Note that \eqref{mode} is a strictly parabolic equation in $\Omega \times(t_0/2, T]$ with smooth coefficients.    Now define $h(x,t) = H( \Psi(x, t))$ on $\Xi_0$ and extend to $\Xi$ by setting it be zero for $0 \le t \le t_0/2$.  Then $h$ is the required solution for (a).  Uniqueness is a consequence of the maximum principle.

We now turn to part (b) of the claim.

\medskip
\noindent
 {\bf Assertion 1.} If $m>0$ is sufficiently large then for any open $E \subset \partial \Omega_0$ there exists $(w, c)\in E \times  (t_0-1/m, t_0)$ such that $\nabla_{V} P_{w, c}(x_0, t_0)\neq  0$.  Here $P_{w, c} (x, t)$ is the solution to the heat equation on $\Omega \times(0, T]$ with zero initial data and boundary data $f=\delta_{(w, c)}$, for $\delta_{w, c}$ the delta function on $\partial \Omega\times (0, T]$ supported at $(w, c)$.

\medskip

 Assertion 1 follows by essentially the same proof as \cite[Postscript]{RR} adapted to our parabolic setting.   Indeed, recall the formula $P_{w, c} (x, t)= \partial_{\nu_{w}} q(x, w, t-c)$ for all $t>c$ (see for example \cite[(4.3.28)]{J}) where $q$ is the Dirichlet heat kernel for $\Omega$ and $\partial_{\nu_{w}}$ is the derivative in the $w$ variable in the direction of the inward facing unit normal.  If the assertion is false then for all $m$ we have $\nabla_{V}  \partial_{\nu_{w}} q(x, w, t_0-c)= \partial_{\nu_{w}} \nabla_{V} q(x, w, t_0-c)=0$ for all $(w, c)\in E \times  (t_0-1/m, t_0)$, where $\nabla_V$ is the derivative in the first and third variables.  Hence $H(w,t):= \nabla_{V} q(x_0, w, t)=0$ on $\partial \Omega \times (0, 1/m)$ while $\partial_{\nu_{w}}  H(w,t) =0$ on $E \times (0, 1/m)$ and it follows by a unique extension result for solutions to the heat equation \cite{LIN} that $H(w, t)=0$ on $\Omega \times (0, 1/m)$, which we show in the following cannot be true, thus establishing Assertion 1 by contradiction.

For any $w$ in $\Omega$ we observe  that  (see for example \cite[Theorem 1.1]{v} and \cite[Section 3.2]{F})
$$\ve^n q(w+\ve(x-w), w, \ve^2t) \to K(x, w, t), \quad \textrm{as } \ve \rightarrow 0,$$
smoothly uniformly for $(x,t)$ in compact subsets of $(\mathbb{R}^n\setminus \{ w\}) \times(0,
\infty)$, where $$K(x,w,t)=\frac{1}{(4\pi t)^{n/2}} e^{-|x-w|^2/4t}$$ is the heat kernel on $\mathbb{R}^n\times(0, \infty)$. From the formula for $K(x, w, t)$ it follows 
that for any fixed $w$ in $\Omega$ there are positive constants $c, C$
 and $T\in (0, 1/m)$ such that in the annulus $A_w$ given by $c<|x-w|<C$
and for $t\in (T/2, T)$ we have:
$$ \frac{\partial q}{\partial t} (x, w, t) >  0, \quad \textrm{and} \quad  \sum_{i=1}^n W^i 
\frac{\partial q}{\partial x_i}(x, w, t)>0$$ where $(W_1, \ldots, W_n) = w-x$ is the vector from $x$ to $w$.

Now let $x_0\in \Omega$ and $V=(V_1, \ldots, V_{n}, V_{n+1})$ be the unit vector as above, recalling that $V_{n+1} \ge 0$.   We may then choose $w\in
\Omega$ such that $x_0\in A_w$ and $(V_1, \ldots, V_n)=\lambda (w-x_0 )$ for some $\lambda \ge 0$. It follows then that
$H(w,t):=\nabla_V q(x_0, w, t)>0$ for some $t\in (0, 1/m)$.  

This completes the proof of Assertion 1 by contradiction.

\medskip
\noindent
{\bf Assertion 2.}  Let $E$ be a strongly convex open subset of $\partial{\Omega_0}$ and  $(a,b) \subset (0,T)$.  
Then $ \Psi_{\alpha,\beta}^{-1}(E\times(a, b))$ contains a strongly convex open subset $E_{\alpha, \beta}$ for all  $\alpha, \beta$ sufficiently close to $0, 1$.

\medskip

 By the strong convexity of $E$ and property (ii) of the map $\Psi_{\alpha, \beta}$, the set $V_{\alpha, \beta}=\Psi^{-1}_{\alpha, \beta}(E \times(a, b))$ is strongly convex in the spatial directions as long as $(\alpha, \beta)$ is sufficiently close to $(0,1)$.  We now show that $V_{\alpha, \beta}$ is \emph{strictly convex}.  Take any pair $(x,s), (y,t) \in V_{\alpha, \beta}$ and consider the line segment $L$ joining them (by shrinking $E$ if necessary, we may assume $L$ is contained entirely in $\Omega \times (0,T]$).  We now show that no interior point of $L$ lies in $V_{\alpha, \beta}$.  By the strong convexity of $V_{\alpha, \beta}$ in the space directions, we may assume that $s \neq t$, in which case the interior of $L$ lies strictly above (i.e. has strictly larger time component) than the interior of the parabolic segment
$$\lambda \mapsto ((1-\lambda) x+ \lambda y, ((1-\lambda)\sqrt{s} + \lambda \sqrt{t})^2), \quad \lambda \in [0,1],$$
connecting $(x,s)$ to $(y,t)$.  Since $\{ u \ge \alpha \}$ is parabolically convex we have $u \ge \alpha$ on this parabolic segment, and since $u_t>0$ in $\Xi$ we have $u>\alpha$ on the interior of $L$ and hence no such point can lie in  $V_{\alpha, \beta}$.  We have thus far shown that $V_{\alpha, \beta}$ is a strictly convex subset of $S_{\alpha}$.  
Assertion 2 follows from the fact that every open strictly convex hypersurface in $\mathbb{R}^{n+1}$ contains an open subset which is strongly convex.  Indeed, after a coordinate rotation we may write such a hypersurface locally as a graph $x_{n+1}=f(x_1,..,x_n)$ over a ball $B \subset \mathbb{R}^n$ such that $f$ attains a minimum value at the center $B$ and is strictly positive on $\partial B$.  By comparing with a quadratic function and applying the maximum principle we obtain that $f$ and hence the hypersurface is strongly convex at some point.  Thus Assertion 2 holds.

We may now complete the proof of part b) of the claim.  Fix a strongly convex open subset $E$ of $\partial{\Omega_0}$ (every smooth convex hypersurface contains such a subset, see for example \cite[p. 104]{RR}) and an interval $(a, b)\subset (t_0-1/m, t_0)$.  By Assertion 1 and shrinking $E$ and $(a, b)$ if necessary, we may assume $|\nabla_{V} P_{(w, c)}(x_0, t_0)|>C>0$ for some some $C>0$ and all $(w, c)\in E\times(a, b)$.  Now let $P^{\alpha, \beta}_{(w, c)}(x, t)$ be the solution to $\eqref{mode}$ on $\Omega \times (0,T]$ with zero initial data and boundary data $\delta_{(w, c)}$.  By property (ii) of the map $\Psi_{\alpha, \beta}$, it follows that for all $\alpha, \beta$ sufficiently close to $0, 1$ we have $|\nabla_{V} P^{\alpha, \beta}_{(w, c)}(x_0, t_0)|>C/2>0 $ for all $(w, c)\in E\times(a, b)$.  The claim then follows by using smooth compactly supported approximations of $\delta_{(w, c)}$, the fact that $H\circ \Psi_{\alpha, \beta}$ solves the standard heat equation on $\Xi$ if $H$ solves $\eqref{mode}$ on $\Omega \times (0,T]$, and Assertion 2.
 \end{proof}

Returning to the proof of Lemma \ref{keylemma}, let $h(x, t)$ be a solution to \eqref{hequation} as in the Claim with boundary data $f:E_{\alpha, \beta}\to (0, \beta-\alpha)$.   Thus $h(x, t)$ satisfies condition (ii) in the Lemma and it remains only to prove condition (i). 

By Proposition \ref{propM}, it suffices to show that $\mathcal{Q}$ is nonpositive at the boundary points of $\Sigma'$.  First suppose that $X=(x,s)$, $Y=(y,t)$ or $(X+Y)/2$  lies in a boundary point of $S_{\alpha}$ or $S_{\beta}$, and $s, t>0$.  There are several cases to consider.

\begin{enumerate}
\item If $X$ and $Y$ lie in $S_{\beta}$ then since $h$ vanishes on $S_{\beta}$ there is nothing to prove since we already know that $S_{\beta}$ is convex.
\item If $X$ and $Y$ lie in $S_{\alpha} \setminus E_{\alpha, \beta}$ then $h$ vanishes at $X$ and $Y$ and we conclude as in Case (1).
\item If $X$ and $Y$ lie in $E_{\alpha, \beta}$ then by the strong convexity of $E_{\alpha, \beta}$ we have
$$\frac{u(x,s) + u(y,t)}{2} - u \left( \frac{x+y}{2}, \frac{s+t}{2} \right)  + c (|x-y|^2+ |s-t|)^2\le 0$$
for $c>0$ sufficiently small.  But $|h(x,s) - h(y,t)|^2 = |f(x, s) - f(y, t)|^2 \le C(|x-y|^2+|s-t|^2)$ for a uniform $C$ depending only on  $E_{\alpha, \beta}$  and $f$.
 Condition (i) follows by replacing $h$ with a sufficiently small multiple of itself.
\item If $X \in E_{\alpha, \beta}$ and $Y \in S_{\alpha} \setminus E_{\alpha,\beta}$, then we may assume $X$ lies in the support of $f$ as otherwise $h(X)=h(Y)=0$ and we may conclude as in case (2).  Under this assumption, that $E_{\alpha, \beta}$ is a strongly convex neighborhood of $S_{\alpha}$  implies that $\frac{1}{2}(X+Y)$ is not in $S_{\alpha}$ and hence $u \left( \frac{X+Y}{2} \right) \ge \alpha+d$ for a uniform constant $d>0$ while $u(X)=u(Y)=\alpha$, and so $\mathcal{Q}(X, Y)\leq 0$ after replacing $h$ with a sufficiently small multiple of itself if necessary.  We argue similarly if the roles of $X,Y$ are reversed.
\item If $(X+Y)/2$ lies in $S_{\alpha}$ then by convexity of $S_{\alpha}$ the points $X$ and $Y$ lie in $S_{\alpha}$ and this reduces to one of the cases above.
\item If $(X+Y)/2$ lies in $S_{\beta}$ then $u(x,s) = u(y,t) = \beta- r$ for some $r \in [0,1)$.  Then note that  $h \le  \beta-u$ by the maximum principle so that after replacing $h$ with a sufficiently small multiple of itself if necessary we have
 \begin{equation*} \begin{split}  (h(x,s) - h(y,t))^2 \le r^2 
\le r
& = \beta - \frac{u(x,s) + u(y,t)}{2}\\
& = u \left( \frac{x+y}{2}, \frac{s+t}{2} \right) - \frac{u(x,s) + u(y,t)}{2} \end{split}\end{equation*}
as required.  
\end{enumerate}

It remains to deal with the case when $s$ or $t$ tends to zero.  Notice that if \emph{both} $s, t$ are less than $t_0/2$ then $h(X)=h(Y)=0$ and $\mathcal{Q} \leq  0$ by the weak convexity of superlevel sets of $u$ already proved.
Hence we may assume without loss of generality that we have a sequence of points $X_i=(x_i, s_i) \rightarrow (x,s)=X$ and $Y_i=(y_i, t_i) \rightarrow (y,t)=Y$ with $s=0$, $t \ge t_0/2$ and 
\begin{equation} \label{contradiction}
\mathcal{Q}(X_i, Y_i) \ge \ve
\end{equation} for some $\ve>0$.  We may assume that $x$ lies in  $\partial \Omega_1$.  Since $(y,t) \in \Xi$ with $t \ge t_0/2$ it follows that $|y-x|$ is bounded below uniformly away from zero.

Using these facts, Lemma \ref{parabolic-convexity-along-parabolas} and fact that $u_t>0$ in $\Omega\times(0, T]$ we may conclude:

\begin{equation*} \begin{split} u(Y)&-u\left( \frac{X+Y}{2}\right)\\&= \left( u(y, t)-u((x+y)/2, t/4)\right) + \left( u((x+y)/2, t/4)-
u((x+y)/2, t/2) \right) \\
&\leq  \left( u((x+y)/2, t/4)-
u((x+y)/2, t/2) \right) \\
&<-c\end{split} \end{equation*}
for some constant $c$ depending only on $\alpha, \beta, t_0$.  Indeed, the first inequality follows from Lemma \ref{parabolic-convexity-along-parabolas} while the second inequality follows from the fact that $u_t>0$ in $\Omega\times(0, T]$ and that $\textrm{dist}((x+y)/2, \partial \Omega)$ is bounded uniformly away from zero depending on  $\alpha, \beta, t_0$.  Thus and after replacing $h$ with a sufficiently small multiple of itself if necessary we obtain $\mathcal{Q}(X,Y) \le 0$, contradicting (\ref{contradiction}).
This completes the proof of the lemma.
\end{proof}

\section{Remarks and open questions} \label{sectionremarks}

Finally, we end with some remarks and open questions related to the results of this paper.

\begin{enumerate}
\item We expect that our proof should carry over to more general parabolic equations (cf. \cite{Ish2}).
\item It would be interesting to know whether superlevel sets of $u$ are \emph{strictly parabolically convex} (with the obvious definition).
\item In view of the explicit solution (\ref{v}) of the heat equation on the half line which is exactly parabolically convex, we expect that the convexity of the superlevel sets of $u$ cannot be sharpened to $p$-convexity for $p>2$  (as defined by taking the functional (\ref{defnC1}) with $p>2$).
\item We used here the parabolic analogue of the two-point function of Rosay-Rudin \cite{RR}.  A related two-point function was introduced in \cite{W} and we expect this also to have a parabolic version.
\item By analogy to the elliptic case, it would be interesting to know whether parabolic ``microscopic'' techniques (cf. \cite{CH, ChS, HM}), analyzing the principal curvatures of the space-time level sets, yield a different proof of Theorem \ref{maintheorem}.
\item  A well-known open problem, mentioned in the introduction, is to extend Borell's result to initial data that is not identically zero.
\end{enumerate}


\begin{thebibliography}{0}

\bibitem{A}  Ahlfors, L.V., Conformal invariants: Topics in geometric function theory, McGraw-Hill Series
in Higher Mathematics. McGraw-Hill Book Co., New York-D\"usseldorf-Johannesburg, 1973.


\bibitem{ALL} Alvarez, O., Lasry, J.-M., Lions, P.-L., {\em Convex
    viscosity solutions and state constraints}, J. Math. Pures
  Appl. (9) 76 (1997), no. 3, 265--288.



\bibitem{BG} Bian, B., Guan, P., {\em 
A microscopic convexity principle for nonlinear partial differential equations}, Invent. Math. 177 (2009), 307--335.


\bibitem{BGMX} Bian, B., Guan, P., Ma, X.N., Xu, L., {\em A constant rank theorem for quasiconcave
solutions of fully nonlinear partial differential equations}, Indiana Univ. Math. J.
60 (2011), no. 1, 101--119.

\bibitem{BLS} Bianchini, C., Longinetti, M., Salani, P., \emph{Quasiconcave solutions to elliptic problems in convex rings}, Indiana Univ. Math. J. 58 (2009), no. 4, 1565--1589. 


\bibitem{Bo} Borell, C., \emph{Brownian motion in a convex ring and quasiconcavity}, Comm. Math. Phys. 86 (1982), no. 1, 143--147.

\bibitem{Bo2} Borell, C., \emph{A note on parabolic convexity and heat conduction}, Ann. Inst. H. Poincar\'e Probab. Statist. 32 (1996), 387--393.

\bibitem{BL}  Brascamp, H.J., Lieb, E.H., {\em On extensions of the Brunn-Minkowski and Pr\'ekopa-Leindler theorems, including inequalities for log concave functions, and with an application to the diffusion equation}, J. Functional Analysis 22 (1976), no. 4, 366--389.

\bibitem{CF} Caffarelli, L., Friedman, A., {\em Convexity of solutions of some semilinear elliptic equations}, Duke Math. J. 52 (1985), no. 2, 431--456.

\bibitem{CGM} Caffarelli, L., Guan, P., Ma, X.N., {\em A constant rank
    theorem for solutions of fully nonlinear elliptic equations},
 Comm. Pure Appl. Math. 60 (2007), no. 12, 1769--1791.

\bibitem{CS} Caffarelli, L., Spruck, J., {\em Convexity properties of
    solutions to some classical variational problems}, Comm. Partial
  Differential Equations 7 (1982), no. 11, 1337--1379.


\bibitem{CMY} Chang, S.-Y.A., Ma, X.-N., Yang, P., \emph{Principal curvature estimates for the convex level sets of semilinear elliptic equations}, Discrete Contin. Dyn. Syst. 28 (2010), no. 3, 1151--1164.
\bibitem{CW} Chau, A.,Weinkove, B., {\em Counterexamples to quasiconcavity for the heat equation},  Int. Math. Res. Not. IMRN (2020), no. 22, 8564--8579.
\bibitem{CH} Chen, C.Q., Hu, B.W., {\em A Microscopic Convexity Principle for Space-time Convex Solutions
of Fully Nonlinear Parabolic Equations}, Acta Math. Sin. (English
Ser.) 29 (2013), no. 4, 651--674.
\bibitem{CMS} Chen, C.Q., Ma, X.-N., Salani, P., {\em  On space-time quasiconcave solutions of the heat equation}, Mem. Amer. Math. Soc. 259 (2019), no. 1244, v+81pp.


\bibitem{ChS} Chen, C.Q., Shi, S.J., {\em Curvature estimates for the level sets of spatial quasiconcave
solutions to a class of parabolic equations}, Science China Mathematics 54 (2011), no. 10,
2063--2080.


\bibitem{DK} Diaz, J.I., Kawohl, B., \emph{On convexity and starshapedness of level sets for some nonlinear elliptic and parabolic problems on convex rings}, J. Math. Anal. Appl. 177 (1993), no. 1, 263--286.

\bibitem{F}  Friedman, A. {\em Partial differential equations of parabolic type}, Prentice-Hall, Inc., Englewood Cliffs, N.J. 1964.

\bibitem{G} Gabriel, R., {\em A result concerning convex level surfaces of 3-dimensional harmonic
functions}, J. London Math. Soc. 32 (1957), 286--294.
\bibitem{Gu} Guan, P., Xu, L.,
\emph{Convexity estimates for level sets of quasiconcave solutions to fully nonlinear elliptic equations}, 
J. Reine Angew. Math. 680 (2013), 41--67. 

\bibitem{HNS} Hamel, F., Nadirashvili, N., Sire, Y., {\em 
Convexity of level sets for elliptic problems in convex domains or convex rings: two counterexamples}, 
Amer. J. Math. 138 (2016), no. 2, 499--527.
\bibitem{HM} Hu, B.W., Ma, X.N., {\em Constant rank theorem of the spacetime convex solutions of heat
equation}, Manuscripta Math., 138 (2012), no. 1-2, 89--118.
\bibitem{Ish} Ishige, K., Salani, P.,  \emph{Is quasi-concavity preserved by heat flow?}, Arch. Math. (Basel) 90 (2008), no. 5, 450--460.

\bibitem{Ish2} Ishige, K., Salani, P., \emph{Parabolic quasi-concavity for solutions to parabolic problems in convex rings}, Math. Nachr. 283 (2010), no. 11, 1526--1548.

\bibitem{Ish3} Ishige, K., Salani, P., \emph{On a new kind of convexity for solutions of parabolic problems}, Discrete Contin. Dyn. Syst. Ser. S 4 (2011), no. 4, 851--864.

\bibitem{J} Jost, J., Partial Differential Equations.  Second edition. Graduate Texts in Mathematics, 214. Springer, New York, 2007.

\bibitem{Ka} Kawohl, B., {Rearrangements and convexity of level sets in PDE}. Lecture Notes in Mathematics, 1150. Springer-Verlag, Berlin, 1985.

\bibitem {Ko2} Korevaar, N.J., {\em Convexity of level sets for solutions to elliptic ring problems}, Comm.
Partial Differential Equations 15 (1990), no. 4, 541--556.

\bibitem{KL} Korevaar, N.J., Lewis, J.L., {\em Convex solutions of certain elliptic equations have constant rank Hessians}, 
Arch. Rational Mech. Anal. 97 (1987), no. 1, 19--32.

\bibitem{LIN} Lin F. H.,   {\em A uniqueness theorem for parabolic equations}, Comm. Pure Appl. Math. 42 (1988), 125-136. 


\bibitem{L}  Lewis, J., {\em  Capacitary functions in convex rings},  Arch. Rational Mech. Anal. 66 (1977), 201--224.

\bibitem{Lo} Longinetti, M., {\em Convexity of the level lines of harmonic functions}, Boll. Un. Mat.
Ital. A 6 (1983), 71--75.


\bibitem{MS} Monneau, R., Shahgholian, H., \emph{Non-convexity of level sets in convex rings for semilinear elliptic problems}, Indiana Univ. Math. J. 54 (2005), no. 2, 465--471.

\bibitem{RR} Rosay, J.-P., Rudin, W., \emph{A maximum principle for sums of subharmonic functions,
and the convexity of level sets}, Michigan Math. J. 36 (1989), no. 1, 95--111.

\bibitem{S} Shiffman, M., {\em On surfaces of stationary area bounded by two circles or convex
curves in parallel planes},  Ann. of Math. (2) 63 (1956), 77--90.

\bibitem{SWYY} Singer, I., Wong, B., Yau, S.T., Yau, S.S.T., {\em An estimate of gap of the first two eigenvalues in the
Schrodinger operator}, Ann. Scuola Norm. Sup. Pisa Cl. Sci. (4) 12 (1985), no. 2, 319--333.

\bibitem{SW}  Sz\'ekelyhidi, G., Weinkove, B., {\em On a constant rank theorem for nonlinear elliptic PDEs}, Discrete Contin. Dyn. Syst. 36 (2016), no. 11, 6523--6532.

\bibitem{v}  van den Berg, M., {\em Heat equation and the principle of not feeling the boundary}, Proc. Roy. Soc. Edinburgh Sect.  A 112 (1989), no. 3-4, 257--262.

\bibitem{Wa} Wang, X.-J., \emph{Counterexample to the convexity of level sets of solutions to the mean curvature
equation},
J. Eur. Math. Soc. (JEMS) 16 (2014), no. 6, 1173--1182.

\bibitem{W} Weinkove, B., \emph{Convexity of level sets and a two-point function}, Pacific J. Math. 295 (2018), no. 2, 499--509.

\end{thebibliography}
\end{document}